\documentclass[leqno,12pt]{amsart}

\usepackage[utf8]{inputenc}
\usepackage[english]{babel}
\usepackage{amsmath}
\usepackage{amsfonts}
\usepackage{amssymb}
\usepackage{amsthm}
\usepackage{graphicx}
\usepackage{epsfig}

\usepackage{ifthen}

\usepackage{cite}
\usepackage[pagewise]{lineno} 
\usepackage{dsfont}
\usepackage[all]{xy}
\usepackage{indentfirst}

\newtheorem{thm}{Theorem}[section]

\newtheorem{prop}[thm]{Proposition}
\newtheorem{lem}[thm]{Lemma}
\theoremstyle{remark}

\theoremstyle{definition}

\newtheorem{ex}[thm]{Example}

\numberwithin{equation}{section}
\numberwithin{thm}{section}

\numberwithin{equation}{section}
\numberwithin{thm}{section}

\newcommand{\tr}{\mathrm{\,trace}\,}

\newcommand{\spa}{\mathrm{span\,}}
\newcommand{\sech}{\mathrm{sech\,}}

\title[Surfaces with positive relative nullity]{Quasi--minimal surfaces of pseudo--Riemannian space forms 
with positive relative nullity}

\author{Burcu Bekta\c s Demirci}
\address{Fatih Sultan Mehmet Vak\i f University, 
Faculty of Engineering, Department of Civil Engineering, 
Beyo\u{g}lu, Istanbul, Turkey, }
\email{bbektas@fsm.edu.tr (Corresponding Author)}
 
\author{Nurettin Cenk Turgay}
\address{Istanbul Technical University, Faculty of Science and Letters, Department of Mathematics, 34469 Maslak, Istanbul, Turkey}
\email{turgayn@itu.edu.tr}

\begin{document}

\begin{abstract} 
In this paper, we consider surfaces in 4--dimensional pseudo--Riemannian 
space--forms with index 2. First, we obtain some of geometrical properties of such surfaces considering  their relative null space. 
Then, we get classifications of quasi--minimal surfaces 
with positive relative nullity.
\end{abstract}

\subjclass[2010]{53B25, 53C40, 53C42}
\keywords{Quasi--minimal surface, positive relative nullity, 
pseudo--Riemannian space forms}

\newenvironment{nouppercase}{
\let\uppercase\relax
\renewcommand{\uppercasenonmath}[1]{}}{}
\begin{nouppercase}
\maketitle
\end{nouppercase}

\section{Introduction}
A submanifolds of a pseudo--Riemannian manifold is said to be \textit{quasi--minimal} if its mean curvature is light--like at every point. Since quasi--minimal submanifolds does not exist in Riemannian manifolds, they have taken attention of many geometers so far (See, for example, \cite{ChenE42Flat,ChenE42PMCV,Ganchev, HaesenOrtega}). When the ambient space is  a Lorentzian space--time, quasi--minimal submanifolds are also called \textit{marginally trapped} in physics literature because they are closely related  with the concept of trapped surfaces, introduced by Roger Penrose in \cite{Penrose65}.

On the other hand, studying  submanifolds by considering their relative null space was initiated by    M. Dajczer and D. Gromoll in \cite{DajczerGro} where they obtained necessary and sufficent conditions for a spherical submanifold to
have positive relative nullity. Recently, the complete classification of 	marginally trapped surfaces 
with positive relative nullity in Lorentzian space--forms was given by B.-Y.Chen and J. Van der Veken in \cite{ChenVeken1}. Further, they proved that there exists no quasi--minimal surface with positive relative nullity  when the ambient space is  a Robertson--Walker space--time with non--constant sectional curvatures, \cite{ChenVeken2}. 

In \cite{Chen}, B.-Y.Chen mentioned some results concerning marginally trapped surfaces
in Lorentzian space forms and in Lorentzian complex space forms
and he also put forward some open problems 
about classification of such surfaces 
in a $4$--dimensional pseudo--Riemannian space forms with index $2$.

The main purpose of this paper is studying quasi--minimal surfaces of 
a $4$--dimensional pseudo--Riemannian space forms with index $2$ 
from in terms of their relative null spaces. In particular, we obtain the complete local classification of quasi--minimal surfaces  
in the pseudo--Euclidean space $\mathbb{E}^4_2$ and  
a pseudo--sphere $\mathbb{S}^4_2$, respectively. In Sect. 2, after we describe the notation that we will use, we give basic facts on quasi--minimal surfaces of a pseudo--Riemannian space forms. 
In Sect. 3, we present our main results.

\section{Preliminaries}
Let $\mathbb E^n_s$ be the pseudo--Euclidean $n$--space 
defined by 
$\mathbb E^n_s=(\mathbb R^n, \hat{g})$, 
where $\hat{g}$
is the canonical metric tensor of index $s$ given by 
$$\hat{g}=\langle\cdot,\cdot\rangle=-\sum\limits_{i=1}^sdx_i\otimes dx_i+\sum\limits_{j=s+1}^ndx_j\otimes dx_j$$
for a Cartesian coordinate system $(x_1,x_2,\hdots,x_n)$ 
of $\mathbb R^n$. 
A non--zero vector $v$ is said to be 
space--like, light--like or time--like 
if $\langle v,v\rangle>0$, $\langle v,v\rangle=0$ 
or $\langle v,v\rangle<0$, respectively.  

We put 
\begin{align*}
\begin{split}
\mathbb S^n_s=&\{{\bf x}\in \mathbb E^{n+1}_s\;|\;
\langle{\bf x},{\bf x}\rangle=1\},\\
\mathbb H^n_s=&\{{\bf x}\in \mathbb E^{n+1}_{s+1}\;|\;
\langle{\bf x},{\bf x}\rangle=-1\}.\\
\end{split}
\end{align*} 
Then, $\mathbb{S}^n_s$ and $\mathbb{H}^n_s$ are pseudo--Riemannian manifolds
of constant sectional curvature $1$ and $-1$ known as a pseudo--sphere and a pseudo--hyper\-bo\-lic space, respectively. 
For a non--zero real number $c>0$, 
we also denote a $n$--dimensional pseudo--Riemannian space form with 
index $s$ and constant sectional curvature $c$ by 
$R^n_s(c)$. It is known that  
 $$R^n_s(c)=\left\{
\begin{array}{cl}
\mathbb E^n_s &\mbox{if $c=0$,} \\
\mathbb S^n_s &\mbox{if $c =1$,}\\
\mathbb H^n_s &\mbox{if $c=-1$.}
\end{array}
\right.
$$ 
Let $\widetilde\nabla$ and $\tilde g$ stand for the Levi--Civita connection 
and the metric tensor of  $R^n_s(c)$, respectively.

\subsection {Pseudo--Riemannian Submanifolds of $R^n_s(c)$}
Consider an isometric immersion  
$f: (\Omega,\check{g})\hookrightarrow  R^n_s(c)$
from an $m$--dimensional pseudo--Riemannian manifold $(\Omega,\check{g})$ and put $M=f(\Omega)$ with the metric $g=f^*(\check{g})$. 
If $\nabla$ denote the Levi--Civita connection of $M$, then for any vector fields 
$X,\ Y\in TM$ and $\xi\in N^f\Omega$, 
the Gauss and Weingarten formulas are given, respectively, by 
\begin{eqnarray}
\label{MEtomGauss} 
\widetilde\nabla_X Y&=& \nabla_X Y + \alpha_f(X,Y),\\
\label{MEtomWeingarten} 
\widetilde\nabla_X \xi&=& -A^f_\xi(X)+\nabla^\perp_X \xi,
\end{eqnarray}
where $N^f\Omega$  stand for the  normal bundle of $f$, $\alpha_f$ is the second fundamental form, 
$A^f_\xi$ is the shape operator 
along the normal direction $\xi$
and $\nabla^\perp$ is the normal connection of $f$. 
Also, $A^f_\xi$ and $\alpha_f$  are related by
\begin{eqnarray}
\label{MinkAhhRelatedby} 
g(A^f_\xi X,Y)=\tilde{g}(\alpha_f(X,Y),\xi).
\end{eqnarray}
The mean curvature vector $H$ of $f$ is defined by 
$H=\frac{1}{m}\tr\;\alpha_f$. 
Note that  $M$ is called quasi--minimal 
if the mean curvature vector $H$ is a light--like at each point of $M$. 

On the other hand, the curvature tensor $R$ of $M$, the normal curvature tensor $R^\perp$ of 
$f$ and  $\alpha_f$ satisfy 
\begin{subequations}
\begin{eqnarray}
\label{MinkGaussEquation} R(X,Y)Z&=&
c(X\wedge Y)Z+A^f_{ \alpha_f(Y,Z)}X-A^f_{\alpha_f(X,Z)}Y,\\
\label{MinkCodazzi} (\bar \nabla_X \alpha_f )(Y,Z)&=&(\bar \nabla_Y \alpha_f )(X,Z),\\
\label{MinkRicciEquation} R^{\perp}(X,Y)\xi&=&\alpha_f(X,A^f_\xi Y)-\alpha_f(A^f_\xi X,Y),
\end{eqnarray}
\end{subequations}
which are called Gauss, Codazzi and Ricci equations, respectively, where  $X\wedge Y$ and $\bar \nabla \alpha_f$ are defined
respectively by 
\begin{eqnarray*}
(X\wedge Y)Z&=&g(Y,Z)X-g(X,Z)Y,\\
(\bar \nabla_X \alpha_f)(Y,Z)&=&\nabla^\perp_X \alpha_f(Y,Z)-\alpha_f(\nabla_X Y,Z)-\alpha_f(Y,\nabla_X Z).
\end{eqnarray*}

The relative null space of $M$ at a point $p$ is defined by 
$$\mathcal{N}_p=\{X_p\in T_pM\;|\;\alpha_f(X_p,Y_p)=0 
\mbox{ for all $Y_p\in T_pM$ }\}.$$
If the dimension of the relative null space $\mathcal{N}_p$ is non-zero for all $p\in M$, then
$M$ is said to have positive relative nullity in $R^n_s(c)$, \cite{ChenVeken1}.

\subsection {Immersions into $\mathbb{S}^4_2$}
Let $f:(\Omega,\check{g})\hookrightarrow\mathbb{S}^4_2$ be an isometric immersion and $i:\mathbb S^4_2\subset\mathbb E^5_2$ 
be the inclusion. Call $\hat f=i\circ f$. 
We denote the Levi--Civita connection of $\mathbb E^5_2$ by 
$\hat{\nabla}$. Then, we have 
$$N^{\hat f}\Omega=i_*\left(N^f\Omega\right)\oplus \mathrm{span\,}\{\hat f\}$$
and shape operators of $f$ and $\hat f$ satisfy
$
A^f_\eta=A^{\hat f}_{i_*\eta}\mbox{whenever $\eta\in N^f\Omega$}$
and 
$
A^{\hat f}_{\hat f}=-\mathrm{I},
$
where $\mathrm{I}$ is the identity.  
Moreover, the second fundamental forms of $f$ and $\hat f$ 
are related by
\begin{equation}
\nonumber
\alpha_{\hat f}(X,Y)= i_*\left(\alpha_{f}(X,Y)\right)-     g (X,Y)\hat f
\end{equation}
whenever $X,Y\in TM$.
Therefore, we have
\begin{equation}
\label{nablaSrelatedby}
\hat\nabla_XY=i_*\left(\widetilde\nabla_XY\right)-g(X,Y)\hat f.
\end{equation}

\section{Surfaces with positive relative nullity in $R^4_2(c)$}
In this section, we obtain complete local classification of quasi--minimal surface with positive relative nullity in a 
$4$--di\-men\-sio\-nal pseudo--Riemannian space forms.
Throughout this section, 
a surface $M$ of $R^4_2(c)$ is defined by $M=f(\Omega)$ and we put $g=f^*(\check g)$
where $f:(\Omega,\check{g})\hookrightarrow(R^4_2(c),\tilde{g})$ is an isometric immersion and $c\in\{-1,0,1\}$.

\begin{lem}
\label{lemma1}
Let $M$ be a quasi--minimal surface
with positive relative nullity in a pseudo--Riemannian space form 
$R^4_2(c)$.  
Then, at each point $p\in M$, 
there exists an orthonormal basis $\{e_1,e_2\}$ 
for the tangent space of $M$
and a pseudo--orthonormal basis $\{e_3,e_4\}$ for  
the normal space of $M$ such that 
\begin{align}
\label{chosenbase}
\begin{split}
g(e_1,e_1) =-g(e_2,e_2)=\varepsilon, \quad &g(e_1,e_2)=0,\\
\tilde{g}(e_3,e_3)=\tilde{g}(e_4,e_4)=0,\quad &\tilde{g}(e_3,e_4)=-1,\\
\alpha_f(e_1,e_1)=\alpha_f(e_1,e_2)=0, \quad&\alpha_f(e_2,e_2)=e_3.
\end{split}
\end{align} 
\end{lem}

\begin{proof}
Assume that $M$ is a quasi--minimal surface with positive relative nullity in $R^4_2(c)$. 
If $\dim\mathcal N_p=2$  for a $p\in M$, 
then $\alpha_f$ vanishes at $p$ which implies $H_p=0$. 
However, this is a contradiction because $M$ is quasi--minimal. 
Therefore, we  have 
$$\dim\mathcal N_p=1 \quad \mbox{for all $p\in M$}.$$ 

On the other hand, if $\mathcal N_p$ is degenerated, 
i.e., $\mathcal N_p=\spa\{X_p\}$ 
for a light--like vector $X_p\in T_pM$, 
then we have $\alpha_f(X_p,Y_p)=0$ 
for any  $Y_p\in T_pM$. 
If $Y_p$ is chosen to be a unique light--like tangent vector at $p$ 
such that $g_p(X_p,Y_p)=-1$,
then we obtain $H_p=-\alpha_f(X_p,Y_p)=0$ 
which yields another contradiction.  
Consequently, there exists a tangent vector field $e_1\in \mathcal N_p$ 
with $g(e_1,e_1)=\varepsilon\in\{-1,1\}$. Let $e_2$  be a unit vector field orthogonal to $e_1$, which implies 
$g(e_2,e_2)=-\varepsilon,\  g(e_1,e_2)=0$. Since $e_1\in \mathcal N_p$, we have $\alpha_f(e_1,e_1)=\alpha_f(e_1,e_2)=0$ and 
$\dim\mathcal N_p=1$ implies $\alpha_f(e_2,e_2)\neq 0$. Now, we define a light--like vector field $e_3$ by
$$e_3=-2\varepsilon H$$
and choose $e_4$ as the unique light--like vector field normal to $M$ 
such that $\tilde g(e_3,e_4)=-1$. Then,
 we obtain $\alpha_f(e_2,e_2)=e_3.$  Hence, we have obtained all of  conditions appearing in \eqref{chosenbase}.    
\end{proof}

Let us assume that $\{e_1, e_2\}$ is an orthonormal frame and 
$\{e_3, e_4\}$ is a pseudo--orthonormal frame on the 
quasi--minimal surface $M$ in $R^4_2(c)$ 
which satisfy
the equation \eqref{chosenbase}. 
With respect to chosen frame field $\{e_1,e_2,e_3,e_4\}$, 
we have 
\begin{align}
\label{levicon_1_R42}
\nabla_{e_i}e_1&=-\varepsilon\omega_{12}(e_i)e_2,\;\;\; \nabla_{e_i}e_2=-\varepsilon\omega_{12}(e_i)e_1\\
\label{levicon_2_R42}
\nabla^{\perp}_{e_i}e_3&=\phi(e_i)e_3,\;\;\;
\nabla^{\perp}_{e_i}e_4=-\phi(e_i)e_4.
\end{align} 
From now on, we denote 
$\phi(e_1)=\omega$ and $\phi(e_2)=\gamma$.

\begin{prop}
\label{SurfR42cProp1}
Let $M$ be a quasi--minimal surface 
of a pseudo--Riemannian space forms $R^4_2(c)$
with positive relative nullity. 
Then, there exists a local coordinate system $(s,t)$ 
defined on a neighborhood of $p\in M$ 
such that the induced metric tensor $g$ of $M$ 
takes the  form
\begin{equation}
\label{SurfR42cProp1Eq1}
g=\varepsilon(ds\otimes ds-\phi^2dt\otimes dt), \qquad \varepsilon=\pm 1.
\end{equation}
Moreover, the vector fields $e_1=\frac{\partial}{\partial s}$ and
$e_2=\frac 1\phi\frac{\partial}{\partial t}$
satisfy
\begin{subequations}\label{SurfR42cProp1Eq2ALL}
\begin{eqnarray}
\label{SurfR42cProp1Eq2a}\widetilde\nabla_{e_1}e_1=0,&\quad&\widetilde\nabla_{e_1}e_2=0\\
\label{SurfR42cProp1Eq2b}\widetilde\nabla_{e_2}e_1=-\omega e_2,&\quad&\widetilde\nabla_{e_2}e_2=-\omega e_1+e_3,\\
\label{SurfR42cProp1Eq2c}\widetilde\nabla_{e_1}e_3=\omega e_3,&\quad&\widetilde\nabla_{e_2}e_3=\gamma e_3, 
\end{eqnarray}
\end{subequations}
where the functions $\phi,\omega$ and $\gamma$ are defined 
by one of following forms:
\begin{itemize}
\item[i.] For $c=0$, 
\begin{equation}
\label{SurfR42cProp1Eq3ALL_1}
\left\{\begin{array}{l}
\phi(s,t)=A(t)(s+m(t))\\
\displaystyle\omega(s,t)=-\frac{1}{s+m(t)}\\
\displaystyle\gamma(s,t)=\frac{\gamma_0(t)}{s+m(t)}-
\frac{m'(t)}{(A(t)(s+m(t)))^2}
\end{array}
\right.
\end{equation}
\item[ii.] For $\varepsilon c=1$, 
\begin{align}
\label{SurfR42cProp1Eq3ALL_2}
\left\{\begin{array}{l}
\phi(s,t)=A(t)\cos{(s+m(t))}\\
\omega(s,t)=\tan{(s+m(t))}\\
\displaystyle\gamma(s,t)=\sec{(s+m(t))}
\left(\gamma_0(t)+\tan{(s+m(t))}
\frac{m'(t)}{A(t)}\right)
\end{array}
\right.
\end{align}
\item[iii.] For $\varepsilon c=-1$, 
\end{itemize}
\begin{align}\label{SurfR42cProp1Eq3ALL_3}
\left\{\begin{array}{l}
\phi(s,t)=A(t)\cosh{(s+m(t))}\\
\omega(s,t)=-\tanh{(s+m(t))}\\
\displaystyle\gamma(s,t)=\sech{(s+m(t))}
\left(\gamma_0(t)-\tanh{(s+m(t))}
\frac{m'(t)}{A(t)}\right)
\end{array}
\right.
\end{align}
for some smooth functions $m$ and $\gamma_0$, where $A$ is  an arbitrarily chosen positive, smooth,  non--vanishing function.
\end{prop}

\begin{proof}
Suppose that $M$ is a quasi--minimal surface 
of a pseudo--Riemannian space forms $R^4_2(c)$ with 
positive relative nullity. 
From Lemma \ref{lemma1}, 
we choose a frame field $\{e_1, e_2, e_3, e_4\}$ which satisfy
the conditions given by \eqref{chosenbase}.
Considering this, we obtain $A^f_{e_3}=0$. 
Calculating the Codazzi equation \eqref{MinkCodazzi} for $X=Z=e_1$, $Y=e_2$ and $X=Z=e_2$, $Y=e_1$,
we get 
$\omega_{12}(e_1)=0\;\mbox{ and }\;
\varepsilon\omega_{12}(e_2)=\omega,$ respectively. 
Combining these equations with \eqref{levicon_1_R42} and \eqref{levicon_2_R42},  
we get the equations in \eqref{SurfR42cProp1Eq2ALL}.

On the other hand, equations \eqref{SurfR42cProp1Eq2a} and \eqref{SurfR42cProp1Eq2b} gives $[e_1,e_2]=\omega e_2$ which implies $[e_1,\phi e_2]=0$ for a non--vanishing function $\phi$ satisfying 
\begin{equation}\label{SurfR42cProp1ProofEq05}
 e_1(\phi)=-\phi\omega. 
\end{equation}
Therefore, there exists a local coordinate system $(s,t)$ such that $e_1=\frac{\partial}{\partial s}$ and  
$\phi e_2=\frac{\partial}{\partial t}$ defined on 
a neighborhood of any point $p\in M$. Consequently, the induced metric tensor $g$ of $M$ 
takes the form given in \eqref{SurfR42cProp1Eq1} and  
the equation \eqref{SurfR42cProp1ProofEq05} turns into
\begin{equation}
\label{SurfR42cProp1ProofEq06}
\phi_s=-\phi\omega. 
\end{equation}
From the Gauss equation \eqref{MinkGaussEquation} and 
the Ricci equation \eqref{MinkRicciEquation},
we obtain
\begin{equation}\label{SurfR42cProp1ProofEq07}
\omega_s=\omega^2+\varepsilon c\;\;\mbox{and}\;\; 
\gamma_s=\omega\gamma+\frac{\omega_t}\phi.
\end{equation}
By  solving equations appearing in \eqref{SurfR42cProp1ProofEq06} and \eqref{SurfR42cProp1ProofEq07} for $c=0$, $\varepsilon c=1$ and $\varepsilon c=-1$, we get \eqref{SurfR42cProp1Eq3ALL_1}, \eqref{SurfR42cProp1Eq3ALL_2} and \eqref{SurfR42cProp1Eq3ALL_3}, respectively, for some smooth functions 
$\gamma_0,m$ and $A$. 
Note that if $\tilde A$ is another smooth, positive, non--vanishing function, then the local coordinate system $(\tilde s, \tilde t)$ defined by 
$\tilde s=s,\ \tilde t=\int_{c}^t\frac{A}{\tilde A}dt$ 
satisfies all conditions of the lemma for $A=\tilde A$. 
Hence, the function $A$ appearing in \eqref{SurfR42cProp1Eq3ALL_1}, \eqref{SurfR42cProp1Eq3ALL_2} and \eqref{SurfR42cProp1Eq3ALL_3} 
can be chosen ar\-bit\-rarily.
\end{proof}

\subsection{Quasi--Minimal Surfaces in $\mathbb E^4_2$}
In this subsection,  
we get the following local classification theorem
for a quasi--minimal surface with positive relative nullity 
in $\mathbb{E}^4_2$.
\begin{thm}
\label{PRNE42ClassThm} 
A quasi--minimal surface
in $\mathbb E^4_2$ has positive relative nullity 
if and only if
it is congruent to one of the followings:
\begin{enumerate}
\item[(i)] A surface given by
\begin{align}
\label{PRNE42Example1YV}
\begin{split}
f(s,t)=&\left(b(t)s+\int _{t_0}^tm(\xi)b'(\xi)d\xi,s\sinh t+\int _{t_0}^tm(\xi) \cosh\xi d\xi, \right.\\
&\left.s\cosh t+\int _{t_0}^tm(\xi) \sinh\xi d\xi,b(t)s+\int _{t_0}^tm(\xi)b'(\xi)d\xi\right),
\end{split}
\end{align}
\item[(ii)] A surface given by
\begin{align}
\label{PRNE42Example2YV}
\begin{split}
f(s,t)=&\left(b(t)s+\int _{t_0}^tm(\xi)b'(\xi)d\xi,s\cosh t+\int _{t_0}^tm(\xi) \sinh\xi d\xi, \right.\\
&\left.s\sinh t+\int _{t_0}^tm(\xi) \cosh\xi d\xi,b(t)s+\int _{t_0}^tm(\xi)b'(\xi)d\xi\right),
\end{split}
\end{align}
\end{enumerate}
where $m(t)$ is a smooth function and the function $b(t)$ is a solution of $b''(t)-b(t)=F(t)$ 
for a non--vanishing smooth function $F(t)$.
\end{thm}

\begin{proof}
Let $M$ be a quasi--minimal surface 
of $\mathbb{E}^4_2$ with positive relative nullity
and consider a local coordinate system 
$(s,t)$ satisfying the conditions given in Proposition \ref{SurfR42cProp1}. 
Without loss of generality, we take $A(t)=1$.

By considering the first equation in \eqref{SurfR42cProp1Eq2a}, we obtain
\begin{equation}
\label{PRNE42ClassThmProofEq01} 
f(s,t)=sB(t)+B_1(t)
\end{equation}
for some smooth $\mathbb R^4$--valued functions 
$B(t)$ and $B_1(t)$.  
Also, 
the equations \eqref{SurfR42cProp1Eq2c} and \eqref{SurfR42cProp1Eq3ALL_1} implies
\begin{equation}
\label{PRNE42ClassThmProofEq03} 
e_3=\frac{F(t)}{s+m(t)}C_0
\end{equation}
for a non--zero constant vector $C_0$, 
where $F(t)$ is a smooth non--vanishing function defined by 
$F(t)=e^{\int_{t_0}^t \gamma_0(\xi)d\xi}$.

On the other hand, 
the second equation in \eqref{SurfR42cProp1Eq2a} gives 
\begin{equation}
\label{PRNE42ClassThmProofEq02} 
B_1'(t)=m(t)B'(t).
\end{equation}
Considering \eqref{SurfR42cProp1Eq1}, \eqref{PRNE42ClassThmProofEq01} and \eqref{PRNE42ClassThmProofEq02}, we get
$e_2=B'(t)$ which implies $\tilde g(B'(t),B'(t))=-\varepsilon$
and   
combining \eqref{SurfR42cProp1Eq3ALL_1}, \eqref{PRNE42ClassThmProofEq03} and  \eqref{PRNE42ClassThmProofEq02} 
with the second equation in \eqref{SurfR42cProp1Eq2b}, 
we get 
\begin{equation}
\nonumber 
B''(t)-B(t)=F(t)C_0.
\end{equation}
The solution of this equation is 
\begin{equation}\label{PRNE42ClassThmProofEq04b} 
B(t)=\cosh t C_1+\sinh t C_2+b(t)C_0
\end{equation} 
for some constant vectors $C_1, C_2\in \mathbb{E}^4_2$, where  
the function $b(t)$ satisfies $b''(t)-b(t)=F(t)$.
By combining \eqref{PRNE42ClassThmProofEq02},
and \eqref{PRNE42ClassThmProofEq04b} with \eqref{PRNE42ClassThmProofEq01}, we obtain
\begin{align}
\label{PRNE42ClassThmProofEq01Re} 
\begin{split}
f(s,t)=&\left(s\cosh t + \int _{t_0}^tm(\xi)\sinh\xi d\xi\right)C_1\\
&+\left(s\sinh t + \int _{t_0}^tm(\xi)\cosh\xi d\xi\right)C_2\\
&+\left(sb(t)+ \int _{t_0}^tm(\xi)b'(\xi)d\xi\right)C_0.
\end{split}
\end{align} 
Since 
$g(e_1,e_1)=\varepsilon$ and $\tilde{g}(e_1,e_3)=\tilde{g}(e_3,e_3)=0$, the equations \eqref{PRNE42ClassThmProofEq03} and \eqref{PRNE42ClassThmProofEq01Re} give
\begin{align}\nonumber 
\begin{split} 
\tilde{g}(C_0,C_1)=\tilde{g}(C_0,C_2)=0,\quad & \tilde{g}(C_0,C_0)=0,\\
\tilde g(C_1,C_1)=-\tilde g(C_2,C_2)=\varepsilon,\quad &\tilde g(C_1,C_2)=0.  
\end{split}
\end{align}
Therefore, up to a suitable isometry of $\mathbb E^4_2$, 
one can choose $C_0=(1,0,0,1)$, $C_1=(0,0,1,0)$, $C_2=(0,1,0,0)$  for the case $\varepsilon=1$. 
In this case, the equation \eqref{PRNE42ClassThmProofEq01Re} turns into \eqref{PRNE42Example1YV} which gives the case (i) of the theorem. 
For $\varepsilon=-1$, up to a suitable isometry of $\mathbb E^4_2$, 
we  choose $C_0=(1,0,0,1)$, $C_1=(0,1,0,0)$, 
$C_2=(0,0,1,0)$. Then, we obtain the case (ii) of the theorem. 
Hence, the necessary condition is proved.

Conversely, it can be shown that both of the isometric immersions given
by \eqref{PRNE42Example1YV} and \eqref{PRNE42Example2YV}
satisfy $\alpha_f(\partial_s,\partial_s)=\alpha_f(\partial_s,\partial_t)=0$ and
$$\quad \alpha_f(\partial_t,\partial_t)=(s+m(t))F(t)(1,0,0,1)\neq 0.$$ 
Hence, the surfaces given in the theorem are  quasi--minimal and they have positive relative nullity. 
\end{proof}

\subsection{Quasi--Minimal Surfaces in $\mathbb S^4_2$}
In this subsection, we classify quasi--minimal 
surfaces of $\mathbb{S}^4_2$ with positive relative nullity,
that is the case $c=1$.
First, we want to present some examples of quasi--minimal surfaces
in $\mathbb S^4_2$.
 
\begin{ex}
\label{PRNS42Example1}
Consider the immersion $f:\Omega\hookrightarrow\mathbb S^4_2$ defined by
\begin{align}
\label{PRNS42Example1YV}
(i\circ f)(s,t)
=\left(b(t)\cos s, \cos s \sinh t,\sin s, \cos s \cosh t,b(t)\cos s\right)
\end{align}
for a smooth function $b(t)$ and put $M=f(\Omega)$, where $\Omega=(0,2\pi)\times \mathbb R$. 
Then, the mean curvature vector of $f$ is
\begin{equation*} 
i_*H=\frac{b''(t)-b(t)}{\cos s}(1,0,0,0,1)
\end{equation*}
which implies that 
$M$ is a quasi--minimal surface of $\mathbb{S}^4_2$
if and only if $b''(t)-b(t)\neq 0$ for all $t\in M$. 
A direct computation yields that $M$ has positive relative nullity. 
\end{ex}

Similar as Example \ref{PRNS42Example1}, we give the following example.

\begin{ex}
\label{PRNS42Example3}
Similar to Example \ref{PRNS42Example1}, the surface $M$ given by 
\begin{align}
\label{PRNS42Example3YV}
(i\circ f)(s,t)=\left(b(t)\cosh s, \sinh{s}, \cosh s\cos t,\cosh s\sin t, b(t)\cosh s\right)
\end{align}
is  quasi--minimal if and only if $b''(t)+b(t)\neq0$ for all $t\in M$. Moreover, $M$ has positive relative nullity. 
\end{ex}

\begin{prop}
\label{PRNS42Example2}
Let $\alpha:(I,-dt^2)\hookrightarrow\mathbb{S}^2_1$
be a time--like curve pa\-ra\-met\-ri\-zed by its arc--length 
with a non--vanishing curvature $\kappa$ and the unit normal vector 
field $N$. Then,
the surface $M$ in $\mathbb{S}^4_2$, defined by 
the immersion $f:J\times I\hookrightarrow\mathbb S^4_2$
given by 
\begin{align}
\label{PRNS42Example2YV}
\begin{split}
\hat{f}(s,t)=&\cos{s}(b(t),\alpha_1(t),\alpha_2(t),\alpha_3(t), b(t))\\
&+\sin{s}\left(\int_{t_0}^t\kappa(\xi) b'(\xi)d\xi, N_{1}(t), N_{2}(t), N_{3}(t), \int_{t_0}^t\kappa(\xi) b'(\xi)d\xi\right)
\end{split}
\end{align} 
is a quasi--minimal surface with positive relative nullity 
if and only if 
$b$ is a smooth function satisfying the condition 
\begin{equation}\label{PRNS42Example2Eq1} 
b''(t)-\kappa(t)\int _{t_0}^t\kappa(\xi)b'(\xi)d\xi-b(t)\neq0
\end{equation} 
for all $t\in I$, where $j\circ\alpha=(\alpha_1,\alpha_2,\alpha_3)$,
$j_*N=(N_1,N_2,N_3)$ and $j:\mathbb S^2_1\subset\mathbb E^3_1$ is the inclusion.
\end{prop}

\begin{proof}
By a direct computation, 
we obtain 
$\widetilde\nabla_{e_1}e_1=\widetilde\nabla_{e_1}e_2=0$ 
which yields that $e_1{}_p\in\mathcal{N}_p$ for any $p\in M$, where $e_1=\partial_s$ and 
$e_2=\frac{1}{\kappa(t)\sin s+\cos s}\partial_t$. 
A further computation yields that
$$i_*\left(h(e_2,e_2)\right)=-2i_*H=\frac{b''(t)-\kappa(t)\int _{t_0}^t\kappa(\xi)b'(\xi)d\xi-b(t)}{\kappa(t)\sin s+\cos s} (1,0,0,0,1).$$ 
Hence, $M$ is quasi--minimal if and only if the condition given in \eqref{PRNS42Example2Eq1} is satisfied.
\end{proof}

\begin{prop}
\label{PRNS42Example4}
Let $\alpha:(I,dt^2)\hookrightarrow\mathbb{S}^2_1$
be a space--like curve pa\-ra\-met\-ri\-zed by its arc--length
with a non--vanishing curvature $\kappa$ and the unit normal vector 
field $N$. Then,
the surface $M$ in $\mathbb{S}^4_2$, defined by 
the immersion $f:J\times I\hookrightarrow\mathbb S^4_2$,
given by 
\begin{align}
\label{PRNS42Exampl42YV}
\begin{split}
\hat{f}(s,t)=&\cosh s\left(b(t),\alpha_1(t),\alpha_2(t),\alpha_3(t),b(t)\right)\\
&+\varepsilon\sinh s\Big(\int_{t_0}^t\kappa(\xi) b'(\xi)d\xi,N_1(t),N_2(t),N_3(t),\int_{t_0}^t\kappa(\xi) b'(\xi)d\xi\Big)
\end{split}
\end{align}
is a quasi--minimal surface with positive relative nullity 
if and only if 
$b$ is a smooth function satisfying the condition 
\begin{equation}\label{PRNS42Example4Eq1} 
b''(t)-\kappa(t)\int _{t_0}^t\kappa(\xi)b'(\xi)d\xi+b(t)\neq0.
\end{equation} 
for all $t\in I$, where $j\circ\alpha=(\alpha_1,\alpha_2,\alpha_3)$,
$j_*N=(N_1,N_2,N_3)$ and $j:\mathbb S^2_1\subset\mathbb E^3_1$ is the inclusion.
\end{prop}

\begin{proof}
Similar as the proof of Proposition \eqref{PRNS42Example2}, 
it can be seen that $M$ has positive relative nullity and 
$$i_*\left(h(e_2,e_2)\right)=2i_*H=\frac{b''(t)-\kappa(t)\int _{t_0}^t\kappa(\xi)b'(\xi)d\xi+b(t)}{\cosh s+\varepsilon\kappa(t)\sinh s} (1,0,0,0,1)$$ 
gives that 
$M$ is a quasi--minimal surface in $\mathbb{S}^4_2$ if and only if
the condition given by the equation \eqref{PRNS42Example4Eq1} is valid. 
\end{proof}

Now, we are ready to prove the following local classification theorem.

\begin{thm}
\label{PRNES42ClassThm}
A quasi--minimal surface in $\mathbb S^4_2$
has positive relative nullity if and only if 
it is congruent to one of the followings:
\begin{enumerate}
\item[(i)] A surface given by \eqref{PRNS42Example1YV}.
\item[(ii)] A surface described in Proposition \ref{PRNS42Example2}.
\item[(iii)] A surface given by \eqref{PRNS42Example3YV}.
\item[(iv)] A surface described in Proposition \ref{PRNS42Example4}.
\end{enumerate}
\end{thm}

\begin{proof}
In order to prove necessary condition, 
assume that $M$ is a quasi--minimal surface with positive relative 
nullity in $\mathbb{S}^4_2$. 
We choose a local coordinate system $(s,t)$ 
which satisfies the equations \eqref{SurfR42cProp1Eq1} and \eqref{SurfR42cProp1Eq2ALL} and define the tangent vector fields $e_1,e_2$ as in Proposition \ref{SurfR42cProp1}.
Using the equations \eqref{nablaSrelatedby} and \eqref{SurfR42cProp1Eq2ALL}, 
we have 
\begin{subequations}
\label{PRNES42ClassThmEq1}
\begin{eqnarray}
\label{PRNES42ClassThmEq1a}
\hat{\nabla}_{e_1}e_1=-\varepsilon \hat{f},
&\quad&\hat{\nabla}_{e_1}e_2=0\\
\label{PRNES42ClassThmEq1b}
\hat{\nabla}_{e_2}e_1=-\omega e_2,
&\quad&\hat{\nabla}_{e_2}e_2=-\omega e_1+e_3+\varepsilon \hat{f},\\
\label{PRNES42ClassThmEq1c}
\hat{\nabla}_{e_1}e_3=\omega e_3,
&\quad&\hat{\nabla}_{e_2}e_3=\gamma e_3.
\end{eqnarray}
\end{subequations}
We are going to study the cases $\varepsilon=1$ and $\varepsilon=-1$ separately. 

\textit{Case (1)} $\varepsilon=1$. Considering the equations \eqref{SurfR42cProp1Eq3ALL_2} and \eqref{PRNES42ClassThmEq1c}, we get 
\begin{equation}
\label{PRNES42ClassThmEq3}
e_3=\frac{F(t)}{\cos{(s+m(t))}}C_0
\end{equation}
where $C_0$ is a non--zero constant vector in $\mathbb{E}^5_2$ and 
$F(t)$ is a smooth function defined by 
$F(t)=e^{\int_{t_0}^t A(\xi)\gamma_0(\xi)d\xi}$. 
Also, $C_0$ is a light--like vector in $\mathbb{E}^5_2$ 
due to the fact that $\tilde{g}(e_3,e_3)=0$.  
Thus, up to isometries of $\mathbb{E}^5_2$, we can choose 
$C_0=(1,0,0,0,1)$.

On the other hand, the first equation in \eqref{PRNES42ClassThmEq1a} gives 
\begin{equation}
\label{PRNES42ClassThmEq2}
\hat{f}(s,t)=\cos{s}B(t)+\sin{s}B_1(t)
\end{equation}
for some smooth $\mathbb{R}^5$--valued functions  $B(t)$ and $B_1(t)$ which satisfy 
\begin{equation}
\label{PRNES42ClassThmEq4}
\sin{m(t)}B'(t)+\cos{m(t)}B_1'(t)=0
\end{equation}
because of the second one of \eqref{PRNES42ClassThmEq1a}. By considering \eqref{PRNES42ClassThmEq4}, we are going to consider subcases $\sin m(t)=0$,  $\cos m(t)=0$ and $\sin m(t) \cos m(t)\neq0$ separately.

\textit{Case(1a)}: $\sin{m(t)}=0$ on $M$. 
Thus, the equation \eqref{PRNES42ClassThmEq4} implies
\begin{equation} \label{PRNES42ClassThmEq8c}
B_1(t)=B_{10}
\end{equation}
for a non--zero constant vector in $B_{10}\in \mathbb{E}^5_2$. 
Without loss of generality, we put $A(t)=1$.  Then, the equation 
$\hat{g}(\hat{f},\hat{f})=1$  and $\tilde{g}(e_1,e_3)=0$ imply
\begin{align}
\begin{split} 
\label{PRNES42ClassThmEq8b}
\tilde g(B(t),B(t))=\tilde g(B_{10},B_{10})=&1,\\ 
\tilde g(B(t),B_{10})=\tilde g(B(t),C_0)=\tilde g(B_{10},C_0)=&0.
\end{split} 
\end{align}  

On the other hand, by combining the equations \eqref{PRNES42ClassThmEq3},  \eqref{PRNES42ClassThmEq2} and \eqref{PRNES42ClassThmEq8c} with the second equation in \eqref{PRNES42ClassThmEq1b}, 
we obtain
$B''(t)-B(t)=F(t)C_0$
which implies
\begin{equation}
\label{PRNES42ClassThmEq8}
B(t)=\cosh{t}C_1+\sinh{t}C_2+b(t)C_0
\end{equation}
for some constant vectors $C_1, C_2\in\mathbb{E}^5_2$ and a smooth  
function $b$ satisfying $b''(t)-b(t)=F(t)$. Because of \eqref{PRNES42ClassThmEq8b} and \eqref{PRNES42ClassThmEq8}, we have 
\begin{align}\label{PRNE42ClassThmEq10} 
\begin{split} 
g(C_1,C_1)=-g(C_2,C_2)=&1, \\
g(C_1,C_2)=g(C_1,C_0)=g(C_2,C_0)=&0,\\  
\end{split}
\end{align}
Therefore, up to a suitable isometry of $\mathbb S^4_2$, we can choose 
\begin{align}
\nonumber
B_{10}=(0,0,1,0,0),\;
C_1=(0,0,0,1,0),\;C_2=(0,1,0,0,0). 
\end{align}
Consequently, the equation \eqref{PRNES42ClassThmEq2} gives \eqref{PRNS42Example1YV}. 
Hence, we have the surface given in the case (i) of the theorem.

\textit{Case(1b)} $\cos{m(t)}=0$ on $M$. 
Similar to Case (1a), we obtain that $M$ is congruent to the surface 
given by $f(\Omega)$, where 
$f:\Omega\hookrightarrow\mathbb S^4_2$ is defined by  
\begin{equation*}
(i\circ f)(s,t)=(b(t)\sin s, \sin s \sinh t,\cos s, \sin s \cosh t,b(t)\sin s).
\end{equation*}
However, this surface is congruent to the surface given in the case (i)  of the theorem.

\textit{Case(1c)} $\sin{m(t)}\neq 0$ and $\cos{m(t)}\neq 0$ on an open subset $\mathcal O$ of  $M$. By shrinking $\mathcal O$ if necessary, we assume $\mathcal O=I\times J$ for some open intervals $I$ and $J$.
Without loss of generality, we choose $A(t)=\sec{m(t)}$.
In this case, the equation \eqref{PRNES42ClassThmEq4} implies 
$B'_1(t)=-\tan{m(t)}B'(t)$ and $B_1'$ does not vanish on $\mathcal O$.  
Then, we have 
\begin{equation}\label{PRNES42ClassThmEq8d}
e_2=B'(t)
\end{equation}
 which implies $g(e_2,e_2)=g(B'(t),B'(t))=-1$. By combining the equations \eqref{PRNES42ClassThmEq3},  \eqref{PRNES42ClassThmEq2} and \eqref{PRNES42ClassThmEq8d} with the second equation in \eqref{PRNES42ClassThmEq1b}, we obtain
\begin{equation}
\label{PRNES42ClassThmEq7_2}
B''(t)=B(t)- \tan m(t)B_1(t)+F(t) \sec m(t)C_0.
\end{equation}

On the other hand, since $\hat g(\hat f,e_3)=0$, \eqref{PRNES42ClassThmEq3} and \eqref{PRNES42ClassThmEq2}  imply
$\tilde g(B_1,C_0)=\tilde g(B,C_0)=0$ which give
\begin{subequations}\label{PRNES42ClassThmCaseacEq1ALL}
\begin{eqnarray}
\label{PRNES42ClassThmCaseacEq1a}B(t)&=&(b(t),\alpha_1(t),\alpha_2(t),\alpha_3(t), b(t)),\\
\label{PRNES42ClassThmCaseacEq1b}B_1(t)&=&(b_{10}(t), b_{11}(t), b_{12}(t), b_{13}(t), b_{10}(t)).
\end{eqnarray}
\end{subequations}
for some smooth functions $b_0$, $\alpha_i$ and $b_{1j}$. Next, we define the smooth curves $\alpha,\gamma:J\to \mathbb E^3_1$ by
\begin{equation}\label{PRNES42ClassThmalpha-gamma}
\alpha=(\alpha_1,\alpha_2,\alpha_3),\quad  \gamma=(b_{11}, b_{12}, b_{13}) 
\end{equation}
Then, from \eqref{PRNES42ClassThmEq7_2} we have
\begin{equation}
\label{PRNES42ClassThmEq7_5} \alpha''(t)=\alpha(t)- \tan m(t)\alpha_1(t).
\end{equation}

Next, by combining \eqref {PRNES42ClassThmEq2} with \eqref{SurfR42cProp1Eq1}, we obtain $g(B(t), B(t))=g(B_1(t),B_1(t))=1$ and 
$g(B'(t),B'(t))=-1$. We consider these equations and \eqref{PRNES42ClassThmCaseacEq1a} to obtain 
$\langle \alpha,\alpha\rangle=1$ and $\langle \alpha',\alpha'\rangle=1$ which yields that $\alpha$ is a time--like curve lying on $\mathbb S^2_1$ 
pa\-ra\-met\-ri\-zed by  its arc--length. Therefore, one can define (spherical) normal $N=(N_1,N_2,N_3)$ and (spherical) curvature $\kappa$ of $\alpha$ by
\begin{equation}
\label{PRNES42ClassThmEq7_3}
\alpha''=\kappa N+\alpha,\;\; N'=\kappa \alpha'.
\end{equation}
 By combining \eqref{PRNES42ClassThmEq7_5} and \eqref{PRNES42ClassThmEq7_3}, we obtain 
\begin{equation}
\label{PRNES42ClassThmCase1cEq3a}
N(t)=\gamma(t),\quad  \kappa(t)=-\tan m(t)
\end{equation}
and by considering this equation and \eqref{PRNES42ClassThmEq4}, from \eqref{PRNES42ClassThmCaseacEq1b}
we obtain 
\begin{equation}\label{PRNES42ClassThmCase1cEq3b}
b_{10}(t)=\int_{t_0}^t\kappa(\xi) b'(\xi)d\xi
\end{equation}
for a constant $t_0$. Note that $\kappa$ does not vanish on $J$ because of the second equation in \eqref{PRNES42ClassThmCase1cEq3a}. By combining \eqref{PRNES42ClassThmCaseacEq1ALL}, \eqref{PRNES42ClassThmCase1cEq3a} and \eqref{PRNES42ClassThmCase1cEq3b} with 
\eqref{PRNES42ClassThmEq2}, we obtain that $\mathcal O$ is congruent to the surface given
by \eqref{PRNS42Example2YV}. Therefore, we have the case (ii) of the theorem.
 
\textit{Case(2)} $\varepsilon=-1$. Similar to Case(1),  
considering the equations \eqref{SurfR42cProp1Eq3ALL_3} and \eqref{PRNES42ClassThmEq1c}, we get 
\begin{equation}
\label{PRNES42ClassThmEq13}
e_3=\frac{F(t)}{\cosh{(s+m(t))}}C_0
\end{equation}
for a light-like constant vector $C_0\in \mathbb{E}^5_2$, where  
$F(t)$ is a non--vanishing smooth function defined by 
$\displaystyle F(t)=e^{\int_{t_0}^tA(\xi)\gamma_0(\xi)d\xi}$.
Since $C_0$ is a light--like constant vector in $\mathbb{E}^5_2$, 
up to isometries 
of $\mathbb{E}^5_2$, we choose $C_0=(1,0,0,0,1)$.

On the other hand, the first equation in  \eqref{PRNES42ClassThmEq1a} 
gives
\begin{equation}
\label{PRNES42ClassThmEq11}
\hat{f}(s,t)=\cosh{s}B(t)+\sinh{s}B_1(t)
\end{equation}
for some smooth $\mathbb{R}^5$--valued functions $B(t)$ and $B_1(t)$.
Also, from the second  equation in \eqref{PRNES42ClassThmEq1a}, 
we obtain 
\begin{equation}
\label{PRNES42ClassThmEq12}
\sinh{m(t)}B'(t)=\cosh{m(t)}B'_1(t).
\end{equation} 
We are going to study the subcases $\sinh m(t)=0$ and $\sinh m(t)\neq 0$ seperately.

\textit{Case(2a)}  $\sinh{m(t)}=0$ on $M$. 
Then, the equation \eqref{PRNES42ClassThmEq12} implies
$B_1(t)=B_{10}$ for a non--zero constant vector $B_{10}\in\mathbb{E}^5_2$. 
Without loss of generality, we put $A(t)=1$. 
Then, by considering $\hat{g}(\hat{f},\hat{f})=1$ and 
$\tilde{g}(e_1, e_3)=0$, from \eqref{PRNES42ClassThmEq13}  and \eqref{PRNES42ClassThmEq11} we get 
\begin{align}
\label{PRNES42ClassThmCase2aEq1}
\begin{split}
\tilde g(B(t),B(t))=-\tilde g(B_{10},B_{10})=1, \quad  &\tilde g(B(t), B_{10})=0,\\
\tilde{g}(B(t),C_0)=\tilde{g}(B_{10}, C_0)=0.&
\end{split}
\end{align}

On the other hand,
equations \eqref{PRNES42ClassThmEq13} and \eqref{PRNES42ClassThmEq11} give 
$$ e_2=B'(t) \quad\mbox{ and }\quad e_3=\frac{F(t)}{\cosh{s}}C_0,$$ 
respectively. 
By combining these equations with the second equation in \eqref{PRNES42ClassThmEq1b}, 
we obtain 
\begin{equation}\nonumber
B''( t)+B( t)=F( t)C_0
\end{equation}
whose solution is
\begin{equation}
\label{PRNES42ClassThmEq16}
B( t)=\cos{ t}C_1+\sin{ t}C_2+b( t)C_{0}
\end{equation}
for some constant vectors in $C_1,C_2\in \mathbb{E}^5_2$, where  
$b$ is a smooth function satisfying 
$b''( t)+b( t)= F( t)$.
 By combining \eqref{PRNES42ClassThmCase2aEq1} and \eqref{PRNES42ClassThmEq16}, we get
\begin{align}\nonumber
\begin{split}
g(C_1,C_0)=g(C_2,C_0)=0 \quad &g(B_{10},C_1)=g(B_{10},C_2)=0\\
g(C_1,C_1)=g(C_2,C_2)=1, \quad &g(C_1, C_2)=0.
\end{split}
\end{align}
Up to isometries of $\mathbb S^4_2$, we choose
$$B_{10}=(0,1,0,0,0),\;\;
C_1=(0,0,1,0,0),\;\;
C_2=(0,0,0,1,0).$$
Then, the equation \eqref{PRNES42ClassThmEq11} becomes \eqref{PRNS42Example3YV}. Hence $M$ is congruent to the surface given in the case (iii) of the theorem.

\textit{Case(2b)}  $\sinh {m(t)}\neq 0$ on an open subset $\mathcal O$ of $M$. Similar to Case (1c), we put
$\mathcal O=I\times J$ and $A(t)=\sech{m(t)}$, where $I$ and $J$ are some open intervals. In this case, because of \eqref{PRNES42ClassThmEq12}, on $\mathcal O$ we have
\begin{equation}
\label{PRNES42ClassThmEq20}
B'_1(t)=\tanh{m(t)}B'(t).
\end{equation}
for vector--valued smooth functions $B(t)$ and $B_1(t)$.  
Consequently, we get $e_2=B'(t)$ and $g(e_2,e_2)=g(B'(t),B'(t))=1$.
Using this and the equation \eqref{PRNES42ClassThmEq13} 
in the second one of \eqref{PRNES42ClassThmEq1b}, 
we get 
\begin{equation}
\label{PRNES42ClassThmEq21}
B''(t)=-B(t)+\tanh{m(t)}B_1(t)+\sech{m(t)}F(t)C_0.
\end{equation}

Similar to Case(1c), $B$ and $B_1$ satisfy \eqref{PRNES42ClassThmCaseacEq1ALL} and we define curves $\alpha,\gamma:J\to \mathbb E^3_1$ as in \eqref{PRNES42ClassThmalpha-gamma}. Then, from \eqref{PRNES42ClassThmEq21} we have
\begin{equation}\label{PRNES42ClassThmEq21b} 
\alpha''(t)=-\alpha(t)+\tanh{m(t)}\gamma(t).
\end{equation}
Note that since $\varepsilon=-1$, \eqref{SurfR42cProp1Eq1} and \eqref{PRNES42ClassThmEq11} implies 
$g(B(t),B(t))=g(B'(t),B'(t))=1$ from which we have $\langle \alpha,\alpha\rangle=\langle \alpha',\alpha'\rangle=1$. 
Therefore, $\alpha$ is a space--like curve lying on $\mathbb{S}^2_1$ 
and it is parametrized by its arc--length. 
Similar to Case (1c), we define $\kappa$ and $N$ by
\begin{align}
\label{PRNES42ClassThmEq22}
\alpha''=\kappa N-\alpha,\quad N'=\kappa \alpha'.
\end{align} 
By combining \eqref{PRNES42ClassThmEq21b} and \eqref{PRNES42ClassThmEq22}, we obtain 
\begin{equation}
\label{PRNES42ClassThmCase2bEq1}
N(t)=\gamma(t),\quad  \kappa(t)=\tanh m(t).
\end{equation}
Since $\sinh {m(t)}\neq 0$ on  $\mathcal O$, $\kappa$ does not vanish on $J$. By considering \eqref{PRNES42ClassThmEq20} and \eqref{PRNES42ClassThmCase2bEq1}, we obtain 
\begin{equation}\label{PRNES42ClassThmCase2bEq2}
b_{10}(t)=\int_{t_0}^t\kappa(\xi) b'(\xi)d\xi
\end{equation}
for a constant $t_0$. By combining \eqref{PRNES42ClassThmCaseacEq1ALL}, \eqref{PRNES42ClassThmCase2bEq1} and \eqref{PRNES42ClassThmCase2bEq2} with \eqref{PRNES42ClassThmEq11}, we obtain that $\mathcal O$ is congruent to the surface given
by \eqref{PRNS42Exampl42YV}. Therefore, we have the case (iv) of the theorem. 
 
We have completed the proof of the necessary condition. Conversely, as we describe in Example \ref{PRNS42Example1}, Example \ref{PRNS42Example3}, Proposition \ref{PRNS42Example2}  and Proposition \ref{PRNS42Example4}, all of the surfaces given in the theorem are quasi--minimal and they have positive relative nullity. 
\end{proof}


\bibliographystyle{amsplain}

\end{document}